\documentclass[english]{article}
\usepackage[T1]{fontenc}
\usepackage[latin9]{inputenc}
\usepackage{geometry}
\usepackage{float}
\usepackage{amsmath}
\usepackage{amssymb}
\usepackage{graphicx}
\usepackage{esint}

\makeatletter
\newtheorem{theorem}{Theorem}
\newtheorem{lemma}[theorem]{Lemma}

\newenvironment{proof}[1][Proof]
 {\begin{trivlist} \item[\hskip \labelsep {\bfseries #1}]}{\end{trivlist}}  \newenvironment{example}[1][Example]
 {\begin{trivlist} \item[\hskip \labelsep {\bfseries #1}]}{\end{trivlist}} 
\newcommand{\qed}{\nobreak \ifvmode \relax \else       \ifdim\lastskip<1.5em \hskip-\lastskip       \hskip1.5em plus0em minus0.5em \fi \nobreak       \vrule height0.75em width0.5em depth0.25em\fi}
\date{}

\makeatother

\usepackage{babel}
\begin{document}

\title{Pair correlation of roots of rational functions with rational generating
functions and quadratic denominators }

\author{Khang Tran\\
Department of Mathematics\\
University of Illinois at Urbana-Champaign\\
 \and Alexandru Zaharescu\\
Department of Mathematics\\
University of Illinois at Urbana-Champaign}

\maketitle
\begin{flushright}
\thanks{Dedicated in honor of Mourad Ismail and Dennis Stanton}
\par\end{flushright}
\begin{abstract}
For any rational functions with complex coefficients $A(z), B(z)$ and $C(z)$,
where $A(z)$, $C(z)$ are not identically zero, we consider the sequence of rational
functions $H_{m}(z)$ with generating
function $\sum H_{m}(z)t^{m}=1/(A(z)t^{2}+B(z)t+C(z))$. We provide an
explicit formula for the limiting pair correlation function of the
roots of $\prod_{m=0}^{n}H_{m}(z)$, as $n\rightarrow\infty$, counting
multiplicities, on certain closed subarcs $J$ of a curve
$\mathcal{C}$ where the roots lie. We give an example where the limiting
pair correlation function does not exist if $J$ contains the endpoints
of $\mathcal{C}$.
\end{abstract}
\footnote{The first author acknowledges support from NSF grant DMS-0838434 \textquotedblright{}EMSW21MCTP:
Research Experience for Graduate Students\textquotedblright{}.%
}%
\footnote{The second author's research was partially supported by NSF grant
DMS-0901621.%
}

\section{Introduction}

The root distribution of certain sequences of rational functions
with complex coefficients
has some peculiar connections with the discriminant of the denominator
of its generating function. In particular, one of the authors \cite{tran}
showed that if the generating function of a sequence of polynomials
is the reciprocal of a quadratic
or cubic polynomial, then the roots of the corresponding sequence
of polynomials form a dense set on an explicit algebraic curve. Moreover,
in several cases, the endpoints of the curve are the roots of the
discriminant of the denominator. For the root distribution of various
sequences of polynomials and their generating functions, see Figure
1, Figure 2 and Figure 3 below. For other connections between discriminants
and the root distribution, see \cite{tran-1}. Our goal here is to
study the pair correlation of the roots of a sequence of rational functions
when the denominator of the generating function is quadratic. Pair
correlation is one of the important statistics concerned with the
local spacing distribution of a sequence, and there are many results
in the literature on the local spacing distribution of a variety of
sequences of interest in number theory (see \cite{axz,axz-1,axz-2,bcz,bcz-1,bocazaha,cz,gallager,hejhal,hooley,katzsarnak,kr,
montgomery,mp,mz,rudnicksarnak,rsz,xz}).

Let $\mathcal{T}$ denote the set of triples $(A(z), B(z), C(z))$ with $A(z), B(z), C(z)$
rational functions of z with complex coefficients, $A(z)$ and $ C(z)$ not
identically zero. For each  $(A(z), B(z), C(z))$ in $\mathcal{T}$ we consider
the sequence of rational functions $H_m(z)$ with generating function
given by
\begin{equation}\label{ES}
\sum_{m=0}^{\infty} H_{m}(z)t^{m}=1/(A(z)t^{2}+B(z)t+C(z)) .
\end{equation}
It is easy to see that there exists a finite set of points in the complex plane which
contains all the poles of all the rational functions $H_m(z)$, $m = 0, 1, \dots$
(the multiplicity of the poles may increase as $m$ increases). By contrast,
if we consider all the zeros of all $H_m(z)$, we obtain an infinite set of complex
numbers which (except in some degenerate cases) is dense on a certain
curve $\mathcal{C}$ which will be described below. This naturally raises 
the question of how these points are distributed along $\mathcal{C}$.

Our goal is to study the pair correlation of the roots of 
$\prod_{m=0}^{n}H_{m}(z)$, counting multiplicities. The rational function
$B^2(z)/A(z)C(z)$ plays an important role in our investigation. In this connection,
let us consider the equivalence relation $\sim$ on $\mathcal{T}$ defined 
such that two elements of $\mathcal{T}$ are equivalent if and only if
the corresponding rational function $B^2(z)/A(z)C(z)$  is the same.
We will say that an element $(A(z), B(z), C(z))$ of $\mathcal{T}$ is in standard form
provided $C(z) = 1$, $A(z)$ and $B(z)$ are polynomials,
and there is no root of $B(z)$ which is simultaneously also 
a double (or higher order) root of $A(z)$.
The last condition above can be restated as saying that 
the greatest common divisor $(B^2(z),A(z))$ of $B^2(z)$ and $A(z)$ is square free.
As examples, the triples $(A(z), B(z), C(z))$
used in Figures 1 and 2 below are in standard form, while the one used
in Figure 3 is not in standard form, since in that example  $(B^2(z),A(z))=z^2$.
It is easy to see that in every equivalence class there is exactly one
element of $\mathcal{T}$ in standard form. Indeed, notice first that
for any $(A(z), B(z), C(z))$ in $\mathcal{T}$,
$(A(z), B(z), C(z))$ and $(A(z)/C(z), B(z)/C(z), 1)$ belong to the same equivalence class.
Also, for any nonzero rational function $E(z)$, the triple
$(E^2(z)A(z)/C(z), E(z)B(z)/C(z), 1)$ belongs to the same equivalence class.
Here one can choose $E(z)$ such that $(E^2(z)A(z)/C(z), E(z)B(z)/C(z), 1)$ is in
standard form. For instance, choose first a polynomial $E(z)$ such that
both $E^2(z)A(z)/C(z)$ and $E(z)B(z)/C(z)$ are polynomials. If
$E(z)B(z)/C(z)$ has a root $\alpha$ which is also a multiple root of  $E^2(z)A(z)/C(z)$,
divide $E(z)$ by $z-\alpha$. Then continue this procedure until all such roots
are eliminated. It follows that any equivalence class contains at least one 
triple in standard form. Next,  if 
$(A_1(z), B_1(z), C_1(z)) \sim (A_2(z), B_2(z), C_2(z))$ and both triples
are in standard form, then $C_1(z) = C_2(z) = 1$ and $B_1^2(z)/A_1(z) = B_2^2(z)/A_2(z)$.
Therefore $B_1^2(z)$ divides $B_2^2(z) A_1(z)$, and since  $(B^2(z),A(z))$ is square free,
this forces $B_1(z)$ to divide $B_2(z)$. Similarly, $B_2(z)$ divides $B_1(z)$, so
$B_1(z) = B_2(z)$ and then $A_1(z) = A_2(z)$. In conclusion, each equivalence class 
contains exactly one triple in standard form. Let us also remark that, given an arbitrary triple 
$(A(z), B(z), C(z))$ in $\mathcal{T}$, the process of finding the unique triple in its equivalence
class which is in standard form described above also provides us with a clear understanding
of how the roots of the corresponding rational functions $H_m(z)$ change. More precisely,
replacing $(A(z), B(z), C(z))$ by $(A(z)/C(z), B(z)/C(z), 1)$ has the effect of multiplying
each $H_m(z)$ by the fixed rational function $C(z)$. Moreover, replacing 
 $(A(z)/C(z), B(z)/C(z), 1)$ by $(E^2(z)A(z)/C(z), E(z)B(z)/C(z), 1)$ has the same effect
 as replacing $t$ by $E(z) t$, that is, has the effect of multiplying each $H_m(z)$
 by $E(z)^m$. In conclusion, the distribution of zeros of the rational functions $H_m(z)$ corresponding
 to a given triple $(A(z), B(z), C(z))$ in $\mathcal{T}$ is exactly the same as the one
 obtained by replacing $(A(z), B(z), C(z))$ by the unique triple in standard form in
 its equivalence class, except at the location of finitely many points in the
 complex plane (where the zeros and poles of C(z) and E(z) above lie).

Taking into account the above discussion, we restrict ourselves in what follows
to study the pair correlation of zeros of $\prod_{k=0}^{m}H_{k}(z)$ for triples
$(A(z), B(z), C(z))$ which are in standard form. Note that in this case all the 
$H_m(z)$ are polynomials, and $H_0(z) = 1$, so we may restrict the above product
to $k\ge 1$. We consider the pair correlation problem on an arbitrary subarc $J$ of the
curve $\mathcal{C}$ on which the roots of $H_{m}(z)$ lie. This restriction
sometimes implies an exclusion of the endpoints of $\mathcal{C}$
from $J$. We will see an explicit example in Section 3 where the
limiting pair correlation function exists on any proper subinterval
which does not contain these endpoints. This function does not exist
if these endpoints belong to $J$. The pair correlation function is
explicitly given in the following theorem.

\begin{theorem}\label{maintheorem}

Let $A(z)$ and $B(z)$ be polynomials in $z$ with complex coefficients,
$A(z)$ not identically zero, such that the greatest common divisor of $B^2(z)$ 
and $A(z)$ is square free, and consider the polynomials
$H_{m}(z)$ with generating function
\[
\sum_{m=0}^{\infty}H_{m}(z)t^{m}=\frac{1}{A(z)t^{2}+B(z)t+1} .
\]
Then the set of roots of all the $H_{m}(z)$
 is dense on a fixed curve $\mathcal{C}$.
Also, let 
\[
h(z)=\frac{B^{2}(z)}{A(z)}
\]
 and let $J$ be a closed subarc of $\mathcal{C}$ such that $h(z)$
is piecewise continuously differentiable on $J$ and the inverse function
exists. The function 
\[
f(t):=h^{-1}(4\cos^{2}\pi t)
\]
 maps a subinterval $I\subset[0,4]$ onto $J$. If $f'(z)\ne0$ on
$I$ then the limiting pair correlation function of the roots of $\prod_{k=1}^{m}H_{k}(z)$,
counting multiplicities, exists on $J$, as $m\rightarrow\infty$,
and is given by 
\begin{equation}
g_{J}(x)=\frac{l(J)}{|I|^{2}}\int_{I}g_{I}\left(\frac{l(J)}{|I||f'(t)|}x\right)\frac{dt}{|f'(t)|},\label{eq:pcJ}
\end{equation}
where 
\begin{equation}
g_{I}(x)=\frac{6}{\pi^{2}x^{2}}\sum_{1\le k\le2x}\sigma(k)\log\frac{2x}{k}\label{eq:pcI}
\end{equation}
and $\sigma$ is the sum of divisors function.

\end{theorem}

The first part of the theorem, stating that the set of roots of all
$H_{m}(z)$ is dense on a fixed curve $\mathcal{C}$ has been established
in \cite{tran}. For the sake of completeness, we will present a proof
in Section 2. The proof of the main part of the theorem involving
the study of pair correlation is given in Section 3. The methodology
of the proof is similar to that used in \cite{axz-1}, and with the
one employed in \cite{axz}, where a more elaborate argument is needed since it addresses the effect of addition on pair correlation functions concerning fractions of bounded height.
Section 4 provides an example of the pair correlation function for
a specific sequence of polynomials.

\section{Distribution of roots on a fixed curve}

In this section we consider the root distribution of the sequence
of polynomials $H_{m}(z)$. We recall the $q$-analogue of the discriminant,
a very useful concept introduced by Mourad Ismail \cite{ismail}.
The $q$-discriminant of a polynomial $P(x)$ of degree $n$ with
the leading coefficient $p$ is 
\[
\mathrm{Disc}_{x}(P(x);q)=p^{2n-2}q^{n(n-1)/2}\prod_{1\le i<j\le n}(q^{-1/2}x_{i}-q^{1/2}x_{j})(q^{1/2}x_{i}-q^{-1/2}x_{j})
\]
where $x_{1},\dots,x_{n}$ are roots of $P(x)$. This $q$-discriminant
equals 0 if and only $x_{i}/x_{j}=q$ for some roots $x_{i},x_{j}$.
In the special case when $q\rightarrow1$, this $q$-discriminant
gives the ordinary discriminant of a polynomial. The following theorem
was established in \cite{tran}.

\begin{theorem}\label{quadratic}

Let
\[
\frac{1}{A(z)t^{2}+B(z)t+1}=\sum H_{m}(z)t^{m},
\]
where $A(z)\ne0$. The roots of $H_{m}(z)$ are dense on the curve
$\mathcal{C}$ defined by the conditions 
\begin{eqnarray*}
\mbox{\ensuremath{\Im}}\frac{B^{2}(z)}{A(z)} & = & 0
\end{eqnarray*}
and 
\[
0\le\Re\frac{B^{2}(z)}{A(z)}\le4
\]
in the complex plane.

\end{theorem}

\begin{proof}

Let $z$ be a root of $H_{m}(z)$ with $A(z)\ne0$. Let $t_{1}=t_{1}(z)$
and $t_{2}=t_{2}(z)$ be the roots of $A(z)t^{2}+B(z)t+1$. If $t_{1}=t_{2}$
then $z\in\mathcal{C}$ since $B^{2}(z)=4A(z)$. We consider $t_{1}\ne t_{2}$.
Using partial fractions, we have 
\begin{eqnarray}
\frac{1}{A(z)t^{2}+B(z)t+1} & = & \frac{1}{A(z)(t-t_{1})(t-t_{2})}\nonumber \\
 & = & \frac{1}{A(z)}\sum\frac{t_{1}^{m+1}-t_{2}^{m+1}}{(t_{1}-t_{2})t_{1}^{m+1}t_{2}^{m+1}}t^{m}.\label{eq:quadraticT}
\end{eqnarray}
So the roots of $H_{m}(z)$ are the roots of $t_{1}=qt_{2}$, where
$q$ is an $(m+1)$-th root of unity and $q\ne1$, $A(z)\ne0$. These
roots are the roots of the $q$-analogue of discriminant
\[
\mathrm{Disc}_{t}(A(z)t^{2}+B(z)t+1;q)=q\left(B^{2}(z)-(q+q^{-1}+2)A(z)\right).
\]
Hence
\begin{eqnarray*}
\frac{B^{2}(z)}{A(z)} & = & q+q^{-1}+2\\
 & = & 2\Re q+2
\end{eqnarray*}
Thus $z\in\mathcal{C}$ since $q$ is an $(m+1)$-th root of unity. 

To show the density of the roots of $H_{m}(z)$, we let $\zeta\in\mathcal{C}$
and $U$ be an open neighborhood of $\zeta$ such that the rational
function $B^{2}(z)/A(z)$ is analytic on $U$. Since the set of $(m+1)$-th
roots of unity is dense on the unit circle as $m\rightarrow\infty$,
the set of values $2\Re q+2$ is dense on the interval $[0,4]$ where
$q$ is an $(m+1)$-th root of unity. By the open mapping theorem,
the map $B^{2}(z)/A(z)$ maps $U$ to an open set containing some
point $2\Re q+2$ for large $m$. Thus there is some point $z\in U$
such that 
\[
\frac{B^{2}(z)}{A(z)}=q+q^{-1}+2
\]
or 
\[
\mathrm{Disc}_{t}(A(z)t^{2}+B(z)t+1;q)=0.
\]
This implies that $q$ is the quotient of the two roots of $A(z)t^{2}+B(z)t+1$.
Hence $z$ is a root of $H_{m}(z)$ by (\ref{eq:quadraticT}).

\end{proof}

The root distribution of various sequences of polynomials with their
generating functions are given by figures below.

\begin{center}
\begin{figure}[H]
\begin{centering}
\includegraphics[scale=0.5]{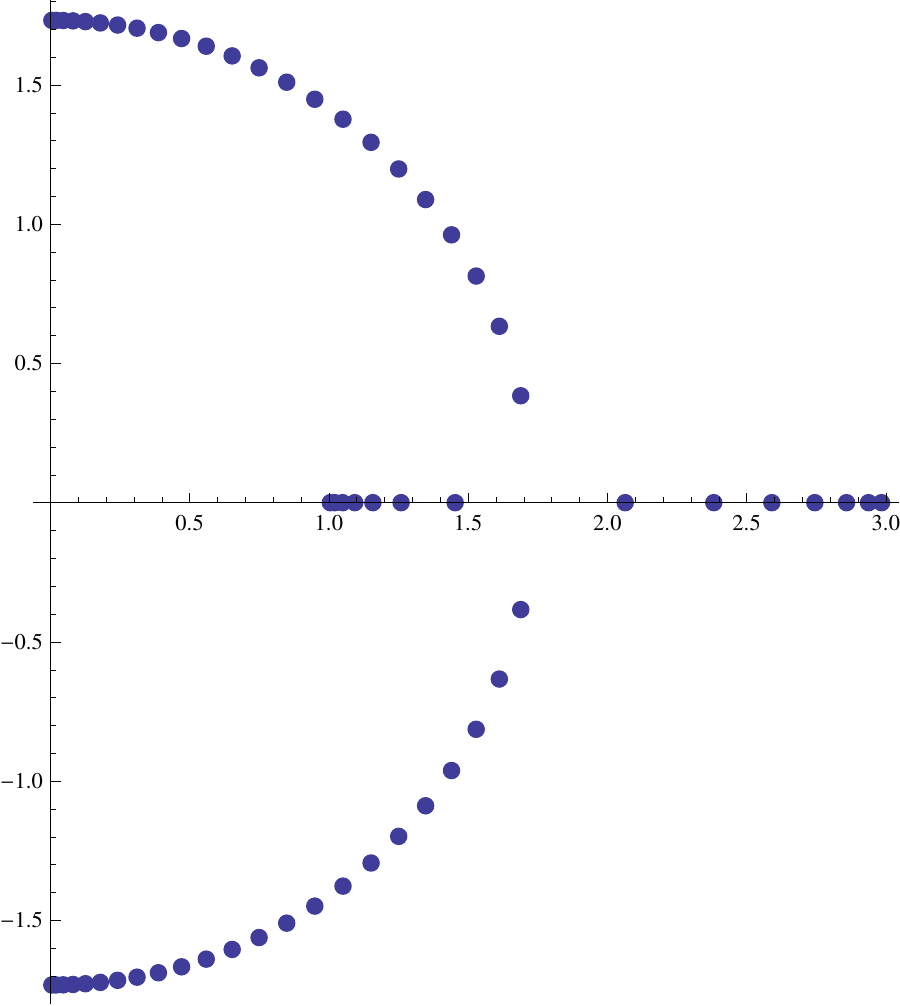}
\par\end{centering}

\caption{Roots of $H_{30}(z)$ with $\sum H_{m}(z)t^{m}=1/(z^{2}t^{2}+(z^{2}-2z+3)t+1)$}
\end{figure}

\par\end{center}

\begin{center}
\begin{figure}[H]
\begin{centering}
\includegraphics[scale=0.5]{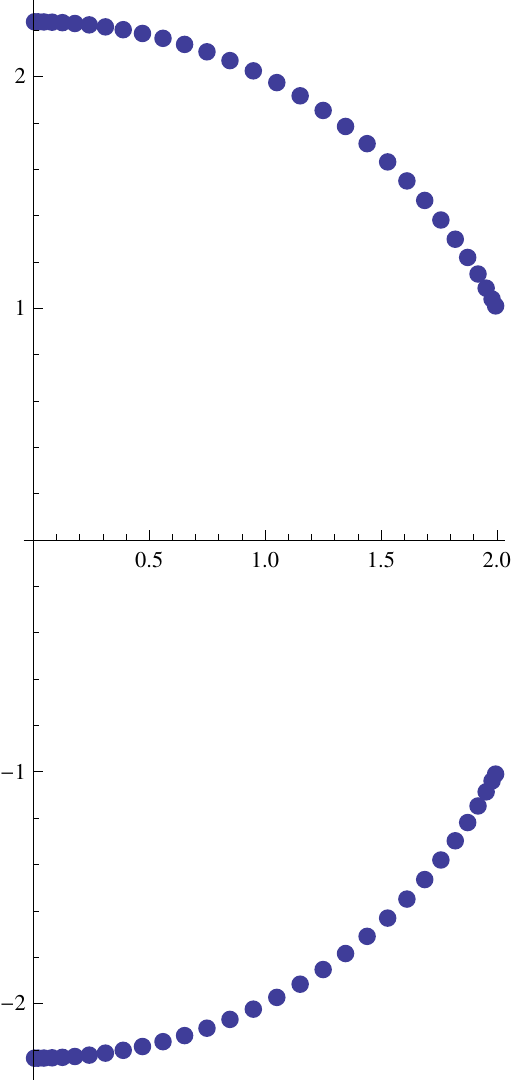}
\par\end{centering}

\caption{Roots of $H_{30}(z)$ with $\sum H_{m}(z)t^{m}=1/(z^{2}t^{2}+(z^{2}-2z+5)t+1)$}
\end{figure}

\par\end{center}

\begin{figure}[H]
\begin{centering}
\includegraphics[scale=0.5]{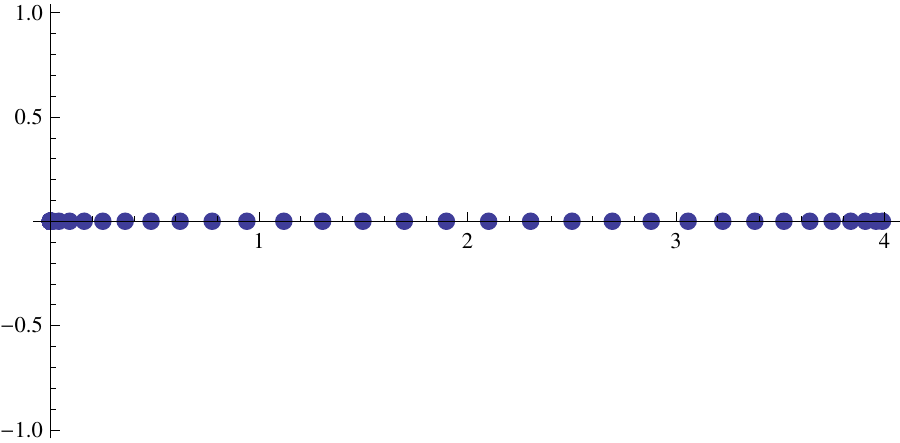}
\par\end{centering}

\caption{Roots of $H_{30}(z)$ with $\sum H_{m}(z)t^{m}=1/(z^{2}t^{2}+(z^{2}-2z)t+1)$}
\end{figure}

In the next section, we will analyze the pair correlation of the roots
of $\prod_{m=1}^{n}H_{m}(z)$, counting multiplicities.

\section{Pair correlation on the corresponding curve}

In this section, we will prove Theorem \ref{maintheorem}. In the
proof of Theorem \ref{quadratic}, the map $h(z)$ maps a root of
$H_{m}(z)$ to $q+q^{-1}+2$ where $q$ is a $(m+1)$-th root of unity.
So it is essential to consider the pair correlation function of the
sequence of finite sequences $M(Q)=\{x_{1},\ldots,x_{N}\}=\{1/2,1/3,2/3,1/4,2/4,3/4,1/5,2/5,3/5,4/5,\ldots\}$
on a subinterval $I=[a,b]$ of $(0,1)$. Each finite sequence is a
sequence of $N=Q(Q-1)/2$ numbers $p/q$ with $1\le p<q$ and $q\le Q$.
Let $M_{I}(Q)=M(Q)\cap I$. The number of points $N_{I}(Q)$ in the
interval $I$ is
\begin{eqnarray*}
N_{I}(Q) & = & \sum_{q\le Q}\sum_{qa\le p\le qb}1\\
 & = & \sum_{q\le Q}q(b-a)+O(1)\\
 & = & \frac{(b-a)Q^{2}}{2}+O(Q)\\
 & = & N(b-a)+O(Q).
\end{eqnarray*}
 We consider the quantity 
\begin{eqnarray*}
R_{M_{I}(Q)}(\lambda) & = & \frac{1}{2N(b-a)}\#\left\{ (x_{i},x_{j})\in M_{I}(Q):i\ne j,0<|x_{i}-x_{j}|\le\frac{\lambda}{N}\right\} .
\end{eqnarray*}
Let the limiting pair correlation measure $R_{M_{I}}(\lambda)=\lim_{Q\rightarrow\infty}R_{M_{I}(Q)}(\lambda)$
if the limit exists. If moreover $R_{M_{I}(Q)}(\lambda)$ can be written
in the form 
\[
R_{M_{I}}(\lambda)=\int_{0}^{\lambda}g_{I}(x)dx,
\]
the function $g_{I}(\lambda)$ is called the pair correlation function
associated to the given sequence of finite sets $\left(M_{I}(Q)\right)_{Q\in\mathbb{N}}$.
We will need the following lemma. The proof of this lemma is similar
to that of Lemma 1 in \cite{axz-1}. 

\begin{lemma}

For any subinterval $I\subset(0,1)$, the pair correlation function
of the sequence of finite sets $\left(M_{I}(Q)\right)_{Q\in\mathbb{N}}$
exists and is given by 
\[
g_{I}(x)=\frac{6}{\pi^{2}x^{2}}\sum_{1\le k\le2x}\sigma(k)\log\frac{2x}{k}
\]
for any $x>0$. 

\end{lemma}

\begin{proof}

Let $G$ and $H$ be smooth functions defined on $\mathbb{R}$
with $\mathrm{Supp}\,G\subset(0,1)$
and $\mathrm{Supp} \,H\subset(0,\Lambda)$. Here
$\mathrm{Supp}\,G$
denotes the closure in $\mathbb{R}$ of the set of points where
$G$ is nonzero, 
and similarly for $\mathrm{Supp} \,H$.
Let 
\begin{eqnarray*}
h(y) & = & \sum_{n\in\mathbb{Z}}H(N(y+n)),\\
g(y) & = & \sum_{n\in\mathbb{Z}}G(y+n)
\end{eqnarray*}
and 
\[
S=\sum_{x,y\in M(Q)}h(x-y)g(x)g(y).
\]
We note that the summations in the definitions of $h(y)$ and $g(y)$
contain only finitely many nonzero terms. This allows us to interchange summations
in the subsequent computations. The functions $h(y)$, $g(y)$ have
Fourier series expansions,
\[
h(y)=\sum_{n\in\mathbb{Z}}c_{n}e(ny)
\]
and
\[
g(y)=\sum_{n\in\mathbb{Z}}a_{n}e(ny).
\]

By substituting these Fourier expansions into the formula of $S$,
we obtain 
\begin{eqnarray*}
S & = & \sum_{x,y\in M(Q)}\sum_{m}c_{m}e(m(x-y))\sum_{n}a_{n}e(nx)\sum_{r}a_{r}e(ry)\\
 & = & \sum_{m,n,r}c_{m}a_{n}a_{r}\sum_{x\in M(Q)}e((m+n)x)\sum_{y\in M(Q)}e((r-m)y)\\
 & = & \sum_{m,n,r}c_{m}a_{n}a_{r}\sum_{q\le Q}\sum_{1\le p<q}e\left(\frac{(m+n)p}{q}\right)\sum_{q\le Q}\sum_{1\le p<q}e\left(\frac{(r-m)p}{q}\right)\\
 & = & \sum_{m,n,r}c_{m}a_{n}a_{r}\left(\sum_{\substack{q\le Q\\
q|m+n
}
}\sum_{1\le p<q}e\left(\frac{(m+n)p}{q}\right)+\sum_{\substack{q\le Q\\
q\nmid m+n
}
}\sum_{1\le p<q}e\left(\frac{(m+n)p}{q}\right)\right)\\
 &  & \times\left(\sum_{\substack{q\le Q\\
q|r-m
}
}\sum_{1\le p<q}e\left(\frac{(r-m)p}{q}\right)+\sum_{\substack{q\le Q\\
q\nmid r-m
}
}\sum_{1\le a<q}e\left(\frac{(r-m)p}{q}\right)\right).
\end{eqnarray*}
Using the fact that for any integer $l$,
\[
\sum_{1\le p\le q}e\left(\frac{lp}{q}\right)=\begin{cases}
0 & \text{if }q\nmid p\\
q & \mbox{if }q|p
\end{cases},
\]
we obtain
\begin{eqnarray*}
S & = & \sum_{m,n,r}c_{m}a_{n}a_{r}\left(\sum_{\substack{q\le Q\\
q|m+n
}
}(q-1)-\sum_{\substack{q\le Q\\
q\nmid m+n
}
}1\right)\left(\sum_{\substack{q\le Q\\
q|r-m
}
}(q-1)-\sum_{\substack{q\le Q\\
q\nmid r-m
}
}1\right)\\
 & = & \sum_{m,n,r}c_{m}a_{n}a_{r}\left(\sum_{\substack{q\le Q\\
q|m+n
}
}q-Q\right)\left(\sum_{\substack{q\le Q\\
q|r-m
}
}q-Q\right)\\
 & = & \sum_{m,n,r}c_{m}a_{n}a_{r}\sum_{\substack{q\le Q\\
q|m+n
}
}q\sum_{\substack{q\le Q\\
q|r-m
}
}q-Q\sum_{m,n,r}c_{m}a_{n}a_{r}\sum_{\substack{q\le Q\\
q|m+n
}
}q\\
 &  & -Q\sum_{m,n,r}c_{m}a_{n}a_{r}\sum_{\substack{q\le Q\\
q|r-m
}
}q+Q^{2}\sum_{m,n,r}c_{m}a_{n}a_{r}.
\end{eqnarray*}

By the Poisson summation formula (see \cite[Theorem 2.2,  page 37]{borwein}
or \cite{guinand} for a generalized version), we have 
\[
\sum_{r\in\mathbb{Z}}a_{r}=\sum_{r\in\mathbb{Z}}g(r)=0
\]
 since $\mathrm{Supp}G\subset(0,1)$. With this fact, $S$ is reduced
to 
\[
S=\sum_{m,n,r}c_{m}a_{n}a_{r}\sum_{\substack{q\le Q\\
q|m+n
}
}q\sum_{\substack{q\le Q\\
q|r-m
}
}q.
\]
With the change in the index of summation given by $m'=m+n$, $n'=r-m$
and $m=r'$, we obtain 
\begin{eqnarray*}
S & = & \sum_{m',n',r'}c_{r'}a_{m'-r'}a_{n'+r'}\sum_{\substack{q\le Q\\
q|m'
}
}q\sum_{\substack{q\le Q\\
q|n'
}
}q\\
 & = & \sum_{q_{1},q_{2}\le Q}q_{1}q_{2}\sum_{\substack{m',n',r'\\
q_{1}|m',q_{2}|n'
}
}c_{r'}a_{m'-r'}a_{n'+r'}\\
 & = & \sum_{q_{1},q_{2}\le Q}q_{1}q_{2}\sum_{r'\in\mathbb{Z}}c_{r'}\sum_{m}a_{mq_{1}-r'}\sum_{n}a_{nq_{2}+r'}.
\end{eqnarray*}
The formula of the coefficients of the Fourier series of $G(t)$ gives
\begin{eqnarray*}
a_{mq_{1}-r} & = & \int_{\mathbb{R}}G(t)e(-(mq_{1}-r)t)dt\\
 & = & \int_{\mathbb{R}}G(t)e(rt)e(-mq_{1}t)dt\\
 & = & \int_{\mathbb{R}}\frac{1}{q_{1}}G\left(\frac{t}{q_{1}}\right)e\left(\frac{rt}{q_{1}}\right)e(-mt)dt.
\end{eqnarray*}
By the Poisson summation formula, 
\[
\sum_{m}a_{mq_{1}-r}=\sum_{m\in\mathbb{Z}}\frac{1}{q_{1}}G\left(\frac{m}{q_{1}}\right)e\left(\frac{rm}{q_{1}}\right).
\]
Following a similar procedure, one obtains
\[
\sum_{m}a_{nq_{2}+r}=\sum_{n\in\mathbb{Z}}\frac{1}{q_{2}}G\left(\frac{n}{q_{2}}\right)e\left(\frac{-rn}{q_{2}}\right).
\]
Applying these two equations to the formula of $S$, we have 
\begin{eqnarray*}
S & = & \sum_{q_{1},q_{2}\le Q}q_{1}q_{2}\sum_{r\in\mathbb{Z}}c_{r}\sum_{m\in\mathbb{Z}}\frac{1}{q_{1}}G\left(\frac{m}{q_{1}}\right)e\left(\frac{rm}{q_{1}}\right)\sum_{n\in\mathbb{Z}}\frac{1}{q_{2}}G\left(\frac{n}{q_{2}}\right)e\left(\frac{-rn}{q_{2}}\right)\\
 & = & \sum_{q_{1},q_{2}\le Q}\sum_{m,n\in\mathbb{Z}}G\left(\frac{m}{q_{1}}\right)G\left(\frac{n}{q_{2}}\right)\sum_{r}c_{r}e\left(\left(\frac{m}{q_{1}}-\frac{n}{q_{2}}\right)r\right)\\
 & = & \sum_{q_{1},q_{2}\le Q}\sum_{m,n\in\mathbb{Z}}G\left(\frac{m}{q_{1}}\right)G\left(\frac{n}{q_{2}}\right)\sum_{r}H\left(N\left(r+\frac{m}{q_{1}}-\frac{n}{q_{2}}\right)\right).
\end{eqnarray*}
Since $\mathrm{Supp}G\subset(0,1)$ , $\mathrm{Supp}H\subset(0,\Lambda)$
and $N\sim Q^{2}/2$, we have that if $r\ne0$ then $G(m/q_{1})G(n/q_{2})H(N(r+m/q_{1}-n/q_{2}))=0$
when $N$ is large. Let $\delta=(q_{1},q_{2})$, $q_{1}=d_{1}\delta$
and $q_{2}=d_{2}\delta$. Let $a,b$ satisfy $ad_{1}+bd_{2}=1$ with
$0<b<d_{1}$. We consider a change in the index of summation given
by $m=bm'+d_{1}n'$ and $n=-am'+d_{2}n'$. This yields
\begin{eqnarray*}
S & = & \sum_{\substack{d_{1}\delta,d_{2}\delta\le Q\\
(d_{1},d_{2})=1
}
}\sum_{m,n\in\mathbb{Z}}G\left(\frac{bm}{d_{1}\delta}+\frac{n}{\delta}\right)G\left(\frac{-am}{d_{2}\delta}+\frac{n}{\delta}\right)H\left(\frac{Nm}{d_{1}d_{2}\delta}\right)\\
 & = & \sum_{\substack{d_{1}\delta,d_{2}\delta\le Q\\
(d_{1},d_{2})=1
}
}\sum_{m,n\in\mathbb{Z}}G\left(\frac{1}{\delta}\left(\frac{bm}{d_{1}}+n\right)\right)G\left(\frac{1}{\delta}\left(\frac{bm}{d_{1}}+n-\frac{m}{d_{1}d_{2}}\right)\right)H\left(\frac{Nm}{d_{1}d_{2}\delta}\right).
\end{eqnarray*}
Since $\mathrm{Supp}H\subset(0,\Lambda)$ and $N\ge(1-\epsilon)Q^{2}/2$
for large $Q$, the summand is nonzero only if 
\[
m\delta\le2\Lambda/(1-\epsilon).
\]
By choosing $\epsilon$ sufficiently small, we have 
\[
m\le2\Lambda/\delta.
\]
This inequality and the fact that $\mathrm{Supp}G\subset(0,1)$ imply
that the summand is nonzero for finitely many values of $n$, which
satisfy
\[
0<\frac{1}{\delta}\left(\frac{bm}{d_{1}}+n\right)<1.
\]
From the fact that 
\[
G\left(\frac{1}{\delta}\left(\frac{bm}{d_{1}}+n-\frac{m}{d_{1}d_{2}}\right)\right)=G\left(\frac{1}{\delta}\left(\frac{bm}{d_{1}}+n\right)\right)+O\left(\frac{m}{\delta d_{1}d_{2}}\right),
\]
we have 
\begin{eqnarray*}
S & = & \sum_{\substack{d_{1}\delta,d_{2}\delta\le Q\\
(d_{1},d_{2})=1
}
}\sum_{\substack{m\delta\le2\Lambda,n}
}G\left(\frac{1}{\delta}\left(\frac{bm}{d_{1}}+n\right)\right)^{2}H\left(\frac{Nm}{d_{1}d_{2}\delta}\right)+O(\log^{2}Q)\\
 & = & \sum_{\substack{m\delta\le2\Lambda,n}
}\sum_{\substack{d_{1},d_{2}\le Q/\delta\\
(d_{1},d_{2})=1
}
}f_{n}\left(\frac{b}{d_{1}}\right)H\left(\frac{Nm}{d_{1}d_{2}\delta}\right)
\end{eqnarray*}
where 
\[
f_{n}(x):=G\left(\frac{1}{\delta}(mx+n)\right)^{2}.
\]
Let $K$ be a large positive integer whose value (as a function of
$Q$) will be given later. The formula of $S$ above becomes 
\begin{eqnarray*}
S & = & \sum_{\substack{m\delta\le2\Lambda\\
n\in\mathbb{Z}
}
}\sum_{i=1}^{K}\sum_{\substack{d_{1},d_{2}\le Q/\delta\\
(d_{1},d_{2})=1\\
b\in[d_{1}i/K,d_{1}(i+1)/K]
}
}f_{n}\left(\frac{b}{d_{1}}\right)H\left(\frac{Nm}{d_{1}d_{2}\delta}\right)+O(\log^{2}Q)\\
 & = & \sum_{\substack{m\delta\le2\Lambda\\
n\in\mathbb{Z}
}
}\sum_{i=1}^{K}\sum_{\substack{d_{1},d_{2}\le Q/\delta\\
(d_{1},d_{2})=1\\
b\in[d_{1}i/K,d_{1}(i+1)/K]
}
}\left(f_{n}\left(\frac{i}{K}\right)+O\left(\frac{1}{K}\right)\right)H\left(\frac{Nm}{d_{1}d_{2}\delta}\right)+O(\log^{2}Q)\\
 & = & \sum_{\substack{m\delta\le2\Lambda\\
n\in\mathbb{Z}
}
}\sum_{i=1}^{K}f_{n}\left(\frac{i}{K}\right)\sum_{\substack{d_{1},d_{2}\le Q/\delta\\
(d_{1},d_{2})=1\\
b\in[d_{1}i/K,d_{1}(i+1)/K]
}
}H\left(\frac{Nm}{d_{1}d_{2}\delta}\right)+O\left(\frac{Q^{2}}{K}\right)+O(\log^{2}Q).
\end{eqnarray*}
We note that if $\min(d_{1},d_{2})\le Q^{1-\epsilon}$ then the main
term is 0 for large $Q$. If $\min(d_{1},d_{2})>Q^{1-\epsilon}$ then
the chain rule implies that $\|DH\|{}_{\infty}\ll1/Q^{1-3\epsilon}$.
Then Lemma 8 in \cite{bcz} gives
\begin{eqnarray*}
\sum_{\substack{d_{1},d_{2}\le Q/\delta\\
(d_{1},d_{2})=1\\
b\in[d_{1}i/K,d_{1}(i+1)/K]
}
}H\left(\frac{Nm}{d_{1}d_{2}\delta}\right) & = & \frac{6}{\pi^{2}K}\iint_{Q^{1-\epsilon}\le x,y\le Q/\delta}H\left(\frac{Nm}{xy\delta}\right)dxdy+O(Q^{3/2+\epsilon})\\
 & = & \frac{6}{\pi^{2}K}\iint_{Q^{1-\epsilon}\le x,y\le Q/\delta}H\left(\frac{Q^{2}m}{2xy\delta}\right)dxdy+O(Q^{3/2+\epsilon})\\
 & = & \frac{6Q^{2}}{\pi^{2}K}\iint_{Q^{-\epsilon}\le x,y\le1/\delta}H\left(\frac{m}{2xy\delta}\right)dxdy+O(Q^{3/2+\epsilon}).
\end{eqnarray*}
If $\min(x,y)\le Q^{-\epsilon}$ then the integrand is 0 when $Q$
is large. Thus the expression becomes
\[
\frac{6Q^{2}}{\pi^{2}K}\iint_{0\le x,y\le1/\delta}H\left(\frac{m}{2xy\delta}\right)dxdy+O(Q^{3/2+\epsilon}).
\]
Putting this expression back into the formula of $S$, we have 
\begin{eqnarray*}
S & = & \frac{6Q^{2}}{\pi^{2}}\sum_{\substack{m\delta\le2\Lambda\\
n\in\mathbb{Z}
}
}\sum_{i=1}^{K}f_{n}\left(\frac{i}{K}\right)\frac{1}{K}\iint_{0\le x,y\le1/\delta}H\left(\frac{m}{2xy\delta}\right)dxdy+O(KQ^{3/2+\epsilon})+O\left(\frac{Q^{2}}{K}\right)\\
 & = & \frac{6Q^{2}}{\pi^{2}}\sum_{\substack{m\delta\le2\Lambda\\
n\in\mathbb{Z}
}
}\int_{0}^{1}f_{n}(x)dx\iint_{0\le x,y\le1/\delta}H\left(\frac{m}{2xy\delta}\right)dxdy+O(KQ^{3/2+\epsilon})+O\left(\frac{Q^{2}}{K}\right).
\end{eqnarray*}
The definition of $f_{n}(x)$ yields 
\begin{eqnarray*}
\sum_{n\in\mathbb{Z}}\int_{0}^{1}f_{n}(x)dx & = & \sum_{n\in\mathbb{Z}}\int_{0}^{1}G\left(\frac{1}{\delta}(mx+n)\right)^{2}dx\\
 & = & \frac{\delta}{m}\sum_{n\in\mathbb{Z}}\int_{n/\delta}^{(m+n)/\delta}G(z)^{2}dz\\
 & = & \delta\int_{0}^{1}G(z)^{2}dz.
\end{eqnarray*}
Choosing $K=[Q^{1/4}]$, we obtain 
\[
S=\frac{6Q^{2}}{\pi^{2}}\left(\int_{0}^{1}G(z)^{2}dz\right)\sum_{m\delta\le2\Lambda}\delta\int_{0\le x,y\le1/\delta}H\left(\frac{m}{2xy\delta}\right)dxdy+O(Q^{7/4+\epsilon}).
\]
Let $\lambda=m/2\delta xy$. The main term is 
\begin{eqnarray*}
 &  & \frac{6Q^{2}}{\pi^{2}}\left(\int_{0}^{1}G(z)^{2}dz\right)\sum_{\delta\le Q}\sum_{m\le2\Lambda/\delta}\delta\int_{0}^{1/\delta}\int_{m/2x}^{\Lambda}H(\lambda)\frac{m}{2\delta x\lambda^{2}}d\lambda dx\\
 & = & \frac{3Q^{2}}{\pi^{2}}\left(\int_{0}^{1}G(z)^{2}dz\right)\sum_{\delta\le Q}\sum_{m\delta/2\le\Lambda}\int_{m\delta/2}^{\Lambda}\int_{m/2\lambda}^{1/\delta}\frac{H(\lambda)m}{x\lambda^{2}}dxd\lambda\\
 & = & \frac{3Q^{2}}{\pi^{2}}\left(\int_{0}^{1}G(z)^{2}dz\right)\sum_{\delta\le Q}\sum_{m\delta/2\le\Lambda}\int_{m\delta/2}^{\Lambda}\frac{H(\lambda)m}{\lambda^{2}}\log\frac{2\lambda}{m\delta}d\lambda.
\end{eqnarray*}
Let $k=m\delta$. The expression becomes
\begin{eqnarray*}
 &  & \frac{3Q^{2}}{\pi^{2}}\left(\int_{0}^{1}G(z)^{2}dz\right)\sum_{0<k\le2\Lambda}\int_{k/2}^{\Lambda}\frac{H(\lambda)}{\lambda^{2}}\log\frac{2\lambda}{k}\sum_{m|k}m\\
 & = & \frac{3Q^{2}}{\pi^{2}}\left(\int_{0}^{1}G(z)^{2}dz\right)\int_{0}^{\Lambda}\frac{H(\lambda)}{\lambda^{2}}\sum_{1\le k\le2\lambda}\sigma(k)\log\frac{2\lambda}{k}d\lambda.
\end{eqnarray*}
We divide the expression by $N(b-a)$ and let $Q\rightarrow\infty.$
The lemma follows after letting $G$ and $H$ approach the characteristic
function of $I$ and $(0,\Lambda)$ respectively.

\end{proof}

We will apply the following theorem which is Theorem 2 from \cite{axz-1}.

\begin{theorem}

Suppose $\mathcal{F}=(\mathcal{F}(Q))_{Q\in\mathbb{N}}$ is a sequence
of finite sequences of points on a closed interval $I$ with $\mathcal{F}(Q)=\{t_{j}^{Q}:1\le j\le N_{Q}\}$.
Let $\mathcal{C}\subset\mathbb{R}^{k}$ be a curve with parametrization
$f:I\rightarrow\mathcal{C}$, where the function $f$ is continuous,
piecewise continuously differentiable, and $f'$ does not vanish in
$I$. Denoting $x_{j}^{Q}=f(t_{j}^{Q})$, we form a sequence of finite
sequences of points on $\mathcal{C}$ by letting $\mathcal{M}(Q)=\{x_{j}^{Q}:1\le j\le N_{Q}\}$
and $\mathcal{M}=\{\mathcal{M}(Q)\}_{Q\in\mathbb{N}}$. Suppose $\mathcal{F}$
is uniformly distributed on $I$ and for any subinterval $I'$ of
$I$, the sequence of functions $R_{\mathcal{F}_{I'}(Q)}(\lambda)$
converges pointwise as $Q\rightarrow\infty$ to a continuous function
$R_{\mathcal{F}_{I'}}(\lambda)$ which is independent of the interval
$I'$. Then the limiting pair correlation measure of $\mathcal{M}$
on the curve $\mathcal{C}$ exists and is given by
\[
R_{\mathcal{M}_{\mathcal{C}}}(\lambda)=\frac{1}{|I|}\int_{I}R_{\mathcal{F}_{I}}\left(\frac{l(\mathcal{C})}{|I||f'(t)|}\lambda\right)dt
\]
where $l(\mathcal{C})$ is the length of the curve $\mathcal{C}$. 

\end{theorem}

The theorem above implies that the limiting pair correlation measure
of roots of $\prod_{k=1}^{m}H_{k}(z)$,
counting multiplicities, exists on $J$ as $m\rightarrow\infty$ and
is given by
\begin{eqnarray*}
\frac{1}{|I|}\int_{I}R_{M_{I}}\left(\frac{l(J)}{|I||f'(t)|}\lambda\right)dt & = & \frac{1}{|I|}\int_{I}\int_{0}^{l(J)\lambda/|I||f'(t)|}g_{M_{I}}(\lambda)d\lambda dt\\
 & = & \frac{l(J)}{|I|^{2}}\int_{0}^{\lambda}\int_{I}g_{M_{I}}\left(\frac{l(J)}{|I||f'(t)|}\lambda\right)\frac{dt}{|f'(t)|}.
\end{eqnarray*}
Hence Theorem \ref{maintheorem} follows.

\section{An explicit example of pair correlation on the corresponding curve}

In this section we will consider the pair correlation for a specific
sequence of polynomials. We have the following result.

\begin{example}

Let $H_{m}(z)$ be a sequence of polynomials satisfying the generating
function 
\[
\sum H_{m}(z)t^{m}=\frac{1}{z^{2}t^{2}+(z^{2}-2z)t+1}.
\]
Then the roots of $H_{m}(z)$ lie on the interval $[0,4)$. The pair
correlation function of the sequence of the roots of $\prod_{m=1}^{n}H_{m}(z)$,
counting multiplicities, on the subinterval $J=[2-2\cos\pi a,2+2\cos\pi a]$
with $0<a<1/2$ exists and is given by
\begin{eqnarray*}
g_{J}(\lambda) & = & \frac{6}{\pi\lambda^{2}\cos\pi a}\sum_{\frac{4\lambda\cos\pi a}{\pi(1-2a)}<k\le\frac{4\lambda\cos\pi a}{\pi(1-2a)\sin a}}\sigma(k)\Biggl(-\sqrt{1-\frac{16\lambda^{2}\cos^{2}\pi a}{\pi^{2}k^{2}(1-2a)^{2}}}+\log\left(1+\sqrt{1-\frac{16\lambda^{2}\cos^{2}\pi a}{\pi^{2}k^{2}(1-2a)^{2}}}\right)\\
 &  & -\log\frac{4\lambda\cos\pi a}{\pi k(1-2a)}\Biggr)+\frac{6}{\pi\lambda^{2}\cos\pi a}\sum_{1\le k\le\frac{4\lambda\cos\pi a}{\pi(1-2a)\sin a}}\sigma(k)\Biggl(\left(1+\log\frac{4\lambda\cos\pi a}{\pi k(1-2a)}\right)\cos\pi a\\
 &  & -\frac{1}{2}\log\frac{1+\cos\pi a}{1-\cos\pi a}-\cos\pi a\log\sin\pi a\Biggr).
\end{eqnarray*}

\end{example}

To show this, let $z=x+iy$. From elementary computations, we note
that 
\begin{eqnarray*}
\Im\frac{B^{2}(z)}{A(z)} & = & \frac{2y(x^{2}+y^{2})P}{(x^{2}+y^{2})^{2}}\\
\mbox{\ensuremath{\Re}}\frac{B^{2}(z)}{A(z)} & = & \frac{P^{2}-Q^{2}}{(x^{2}+y^{2})^{2}}
\end{eqnarray*}
where 
\begin{eqnarray*}
P & = & 2x^{2}+x^{3}-2y^{2}+xy^{2}\\
Q & = & y(x^{2}+y^{2}).
\end{eqnarray*}

From Theorem \ref{quadratic}, we have two cases, $y=0$ or $P=0$.
The condition 
\[
\frac{P^{2}-Q^{2}}{(x^{2}+y^{2})^{2}}\ge0,
\]
implies $y=0$ in both cases. The inequality 
\[
\frac{P^{2}-Q^{2}}{(x^{2}+y^{2})^{2}}\le4
\]
gives $x^{4}(x^{2}-4x)\le0$. This implies that the roots of $H_{m}(z)$
lies on the real interval $[0,4]$. The roots of $H_{30}(z)$ are
given by Figure 3.

The definition of $h(z)$ in Theorem \ref{maintheorem} gives 
\[
h(z)=\frac{B^{2}(z)}{A(z)}=(z-2)^{2}.
\]
The formula $f(t):=h^{-1}(4\cos^{2}\pi t)$ in Theorem \ref{maintheorem}
becomes 
\[
f(t)=2+2\cos\pi t.
\]
This function maps the interval $I=(a,1-a)$ onto $J$ with 
\[
f'(t)=-2\pi\sin\pi t.
\]
Equation (\ref{eq:pcJ}) gives

\begin{eqnarray*}
g_{J}(\lambda) & = & \frac{|J|}{|I|^{2}}\int_{I}g_{I}\left(\frac{|J|}{|I||f'(t)|}x\right)\frac{dt}{|f'(t)|}\\
 & = & \frac{24}{\pi|J|\lambda^{2}}\int_{a}^{1/2}\sum_{1\le k\le|J|\lambda/\pi|I|\sin\pi t}\sigma(k)\sin\pi t\log\frac{|J|\lambda}{|I|k\sin\pi t}\\
 & = & \frac{24}{\pi|J|\lambda^{2}}\sum_{1\le k\le|J|\lambda/\pi|I|}\sigma(k)\int_{a}^{1/2}\sin\pi t\log\frac{|J|\lambda}{|I|k\sin\pi t}dt\\
 &  & +\frac{24}{\pi|J|\lambda^{2}}\sum_{|J|\lambda/\pi I<k\le|J|\lambda/\pi|I|\sin a}\sigma(k)\int_{a}^{t_{k}}\sin\pi t\log\frac{|J|\lambda}{|I|k\sin\pi t}dt
\end{eqnarray*}
where $0\le t_{k}\le1/2$ is the solution to the equation 
\[
\sin\pi t_{k}=\frac{|J|\lambda}{|I|\pi k}.
\]
The integrand in this formula has an antiderivative 
\[
-(1+\log\frac{|J|\lambda}{\pi|I|k})\cos\pi t+\frac{1}{2}\log(\cos\pi t+1)-\frac{1}{2}\log(1-\cos\pi t)+\cos\pi t\log\sin\pi t.
\]
Thus the pair correlation function $g_{J}(\lambda)$ becomes

\begin{eqnarray*}
\frac{24}{\pi|J|\lambda^{2}}\sum_{|J|\lambda/\pi I<k\le|J|\lambda/\pi|I|\sin a}\sigma(k)\left(-\sqrt{1-\frac{|J|^{2}\lambda^{2}}{\pi^{2}|I|^{2}k^{2}}}+\log\left(1+\sqrt{1-\frac{|J|^{2}\lambda^{2}}{\pi^{2}|I|^{2}k^{2}}}\right)-\log\frac{|J|\lambda}{|I|\pi k}\right)\\
+\frac{24}{\pi|J|\lambda^{2}}\sum_{1\le k\le|J|\lambda/\pi|I|\sin a}\sigma(k)\left((1+\log\frac{|J|\lambda}{\pi|I|k})\cos\pi a-\frac{1}{2}\log\frac{1+\cos\pi a}{1-\cos\pi a}-\cos\pi a\log\sin\pi a\right).
\end{eqnarray*}
The corollary follows by noticing that $|J|=4\cos\pi a$ and $|I|=1-2a$. 

We notice that when $a$ approaches 0, the series in the formula of
$g_{J}(\lambda)$ behaves similar to 
\[
\sum_{k=1}^{\infty}\frac{\sigma(k)}{k^{2}}
\]
which diverges. Hence the limiting pair correlation function does
not exist on the whole interval $[0,4]$.


\begin{thebibliography}{References}
\bibitem[1]{axz}E. Alkan, M. Xiong, A. Zaharescu, Pair correlation
of sums of rationals with bounded height, J. Reine Angew. Math. 641
(2010), 21-67.

\bibitem[2]{axz-1}E. Alkan, M. Xiong, A. Zaharescu, Pair correlation
of torsion points on elliptic curves, J. Math. Anal. Appl. 356 (2009),
no. 2, 752\textendash{}763.

\bibitem[3]{axz-2}E. Alkan, M. Xiong, A. Zaharescu, Local spacings
along curves, J. Math. Anal. Appl. 329 (2007), 721\textendash{}735.

\bibitem[4]{bcz}F. P. Boca, C. Cobeli, A. Zaharescu, A conjecture
of R. R. Hall on Farey points, J. Reine Angew. Math. 535 (2001), 207\textendash{}236.
MR1837099.

\bibitem[5]{bcz-1}F. P. Boca, C. Cobeli, A. Zaharescu, Distribution
of lattice points visible from the origin, Comm. Math. Phys. 213 (2000),
433\textendash{}470.

\bibitem[6]{bocazaha}F.P. Boca, A. Zaharescu, The correlations of
Farey fractions, J. London Math. Soc. (2) 72 (2005), 25\textendash{}39.

\bibitem[7]{borwein}J. M. Borwein, P. B. Borwein, Pi and the AGM,
Canadian Mathematical Society Series of Monographs and Advanced Texts,
Wiley-Interscience Publication, 1987.

\bibitem[8]{cz}C. Cobeli, A. Zaharescu, On the distribution of primitive
roots (mod p), ActaArith. 83 (1998), 143\textendash{}153.

\bibitem[9]{gallager}P. X. Gallagher, On the distribution of primes
in short intervals, Mathematika 23 (1976), 4\textendash{}9.

\bibitem[10]{guinand}A.P. Guinand, On Poisson's summation formula,
Ann. Math. (2) 42 (1941), 591\textendash{}603. 

\bibitem[11]{hejhal}D. A. Hejhal, On the triple correlation of zeros
of the zeta function, Internat. Math. Res. Notices 7 (1994).

\bibitem[12]{hooley}C. Hooley, On the intervals between consecutive
terms of sequences, Proc. Symp. Pure Math. 24 (1973), 129\textendash{}140.

\bibitem[13]{ismail}M. E. H. Ismail, Difference equations and quantized
discriminants for $q$-orthogonal polynomials, Advances in Applied
Mathematics, 30 (2003), 562\textendash{}589.

\bibitem[14]{katzsarnak}N. M. Katz, P. Sarnak, Zeros of zeta functions
and symmetry, Bull. Amer. Math. Soc. 36 (1999), 1-26.

\bibitem[15]{kr}P. Kurlberg, Z. Rudnick, The distribution of spacings
between quadratic residues, Duke Math. J. 100 (1999), 211\textendash{}242.

\bibitem[16]{montgomery}H. L. Montgomery, The pair correlation of
zeros of the zeta function, Analytic number theory, Proceedings of
Symposium on Pure Mathematics - Vol. XXIV, St Louis University, MO,
1972 (American Mathematical Society, Providence, RI, 1973), 181-193.

\bibitem[17]{mp}R. Murty, A. Perelli, The Pair Correlation of Zeros
of Functions in the Selberg Class, IMRN (1999), no. 10, 531\textendash{}545

\bibitem[18]{mz}R. Murty, A. Zaharescu, Explicit formulas for the
pair correlation of zeros of functions in the Selberg class, Forum
Math. 14 (2002), no. 1, 65\textendash{}83.

\bibitem[19]{rudnicksarnak}Z. Rudnick, P. Sarnak, Zeros of principal
$L$-functions and random matrix theory, Duke Math. J. 81 (1996),
269-322.

\bibitem[20]{rsz}Z. Rudnick, P. Sarnak, A. Zaharescu, The distribution
of spacings between the fractional parts of $n^{2}\alpha$, Invent.
Math. 145 (2001), no. 1, 37\textendash{}57.

\bibitem[21]{xz}M. Xiong, A. Zaharescu, Pair correlation of rationals
with prime denominators, J. Number Theory 128 (2008), no. 10, 2795\textendash{}2807. 

\bibitem[22]{tran}K. Tran, Connections between discriminants and
the root distribution of polynomials with rational generating functions,
PhD thesis, University of Illinois, 2012.

\bibitem[23]{tran-1}K. Tran, Discriminants of Polynomials Related
to Chebyshev Polynomials: The \textquotedbl{}Mutt and Jeff\textquotedbl{}
Syndrome. J. Math. Anal. Appl. 383 (2011), 120\textendash{}129.\end{thebibliography}
\end{document}